\documentclass{jams-l}
\newcommand{\R}{\mathbb{R}}
\newcommand{\E}{\mathbb{E}}
\def\ceil#1{\lceil #1 \rceil}
\def\floor#1{\lfloor #1 \rfloor}

\usepackage{enumerate}
\usepackage{graphicx}
\usepackage{xcolor}

\usepackage{hyperref}       

\newtheorem{theorem}{Theorem}
\newtheorem*{theorem*}{Theorem}
\newtheorem{lemma}[theorem]{Lemma}

\newtheorem{definition}[theorem]{Definition}
\newtheorem{claim}[theorem]{Claim}

\newtheorem*{remark*}{Remark}

\numberwithin{equation}{section}

\begin{document}

\title[Essential covers]{A lower bound for essential covers of the cube}


\author{Gal Yehuda}
\address{Department of Computer Science, Technion-IIT}
\email{ygal@cs.technion.ac.il}

\author{Amir Yehudayoff}
\address{Department of Mathematics, Technion-IIT}
\email{yehudayoff@technion.ac.il}

\begin{abstract}
Essential covers were introduced by Linial and Radhakrishnan
as a model that captures two complementary properties:
(1) all variables must be included and (2) no element is redundant.
In their seminal paper, they proved that every essential cover of the $n$-dimensional hypercube
must be of size at least $\Omega(n^{0.5})$.
Later on, this notion found several applications in complexity theory.
We improve the lower bound to $\Omega(n^{0.52})$,
and describe two applications.
\end{abstract}

\maketitle

\section{Introduction}

The vertices of the hypercube can be embedded in Euclidean space
as $\{\pm 1\}^n$.
A hyperplane $h = \{z \in \R^n : \langle z, v \rangle = \mu \}$ covers the vertex $x$ if $x \in h$. 
{\em What is the minimum number of hyperplanes that are needed
to cover all vertices?}
This question is motivated by problems in combinatorial geometry~\cite{littlewood1939number},
combinatorics~\cite{alon1993covering,linial2005essential},
proof complexity~\cite{beame2017stabbing,dantchev2021depth} and more.
For an introduction to this and related topics, see the survey by Saks~\cite{saks_1993}.

The two hyperplanes $\{z_1 = 1 \}$ and $\{z_1=-1\}$ cover all vertices.
This solution is optimal but also degenerate and not so interesting.
There are two standard ways for defining non-degenerate covers:
skew and essential covers.
In a skew cover, all normal vectors have full support (see e.g.~\cite{saks_1993,yehuda2021slicing}
and references within).
In an essential cover, all variables must appear
and no hyperplane is redundant~\cite{linial2005essential}.
Our focus is on the latter.


\begin{definition}[Linial and Radhakrishnan~\cite{linial2005essential}]
The hyperplanes $h_1, \ldots, h_k$ form an essential cover of the $n$-cube if
\begin{enumerate}[({E}1)]
    \item for every vertex $x \in \{\pm 1\}^n$, there is $i \in [k]$ so that $x \in h_i$.\label{essential-condition:cover}
    \item for every $j \in [n]$, there is $i \in [k]$ such that the normal $v_i$ of $h_i$ satisfies $v_{i,j} \neq 0$.\label{essential-condition:all-vars}
   \item for every $i \in [k]$, there is a vertex $x$ so that $x$ is covered only by $h_i$. \label{essential-condition:minimal}
 \end{enumerate}
\end{definition}

Condition~(E\ref{essential-condition:cover}) means that $h_1,\ldots,h_k$
form a cover.
Condition~(E\ref{essential-condition:all-vars}) says that every variable
``appears'' in at least one of the hyperplanes.
Condition~(E\ref{essential-condition:minimal}) means that no strict subset
of $h_1,\ldots,h_k$ form a cover (``criticality'' or ``minimality'').

Linial and Radhkrishnan proved that the smallest size $e(n)$ of an essential cover of the $n$-cube satisfies 
\begin{equation*}
    \tfrac{1}{2} (\sqrt{4n + 1} + 1) \leq e(n) \leq \ceil{\tfrac{n}{2}} .
\end{equation*}
When the hyperplanes are assumed to be {\em positive},
a sharp $n+1$ lower bound was proved by Saxton~\cite{SAXTON2013971}.
Saxton also reduced the general case to a problem 
on permanents of matrices.
We provide the first improvement over the lower bound from~\cite{linial2005essential}.

\begin{theorem*}
$e(n) \geq \Omega(n^{0.52})$.
\end{theorem*}

Our proof-plan is straightforward.
Show that few hyperplanes must miss a vertex of the cube.
But this is false; two hyperplanes can cover the whole cube.
We must use the essential condition.
{\em How?}
Our high-level approach is inspired by ideas from~\cite{yehuda2021slicing}
where a lower bound for the slicing problem was proved.
Specifically, 
we rely on the connection to convex
geometry that was established in~\cite{yehuda2021slicing}.
The lower bound is based on Bang's solution
of Tarski's plank problem~\cite{bang1951solution,ball1991plank}.

Proving lower bounds for the slicing problem is in general harder than for  covering problems.
For example, any lower bound for the slicing problem
immediately implies a lower bound for the skew covering problem~\cite{yehuda2021slicing}.
In the slicing problem, however, we can assume that the normal vectors
are generic. This assumption is powerful, e.g., when proving anti-concentration results.
In essential covers, a major difficulty is that the hyperplanes are not generic;
a small perturbation of a cover is not a cover.
This leads to several new challenges,
and requires developing a better structural understanding
of the few but unknown hyperplanes.

\subsection{Applications}
The theorem immediately implies two new lower bounds in proof complexity,
as we briefly describe next.

The first lower bound is for
resolution over linear equations $\text{Res}(\text{lin}_{\mathbb{Q}})$.
This proof system was introduced by Raz and Tzameret \cite{raz2008resolution} as an extension of resolution with ``counting capabilities''.
Part and Tzameret~\cite{part2021resolution} studied 
tree-like proofs in this system.
They proved the following lower bound.
Let $g_1,\ldots,g_N$ be linear functions in $n$ variables over $\mathbb{Q}$ where each depends on at least $\frac{n}{2}$ variables.
Then, any tree-like $\text{Res}(\text{lin}_{\mathbb{Q}})$ derivation of
a tautology of the form $\vee_{j \in [N]} g_j = 0$ is of size $2^{\Omega(\sqrt{n})}$.
The $\Omega(\sqrt{n})$ in the exponent is in fact $\Omega(e(n))$
and follows from Linial and Radhkrishnan's
lower bound.
Our result immediately improves the lower bound
to~$2^{\Omega(n^{0.52})}$.

The second application is in the stabbing planes proof system.
This proof system was introduced by Beame et al.\
as an extension of cutting planes that is more similar to 
algorithms that are used in practice~\cite{beame2017stabbing}.
Dantchev, Galesi, Ghani, and Martin proved a lower bound of $\Omega(n)$ on the refutation-size of Tseitin formulas over the $n \times n$ grid~\cite{dantchev2021depth}.
Their lower bound is in fact of the form $\Omega(e(n^2))$.
Our result thus yields the better lower bound $\Omega(n^{1.04})$.

\section{Finding the uncovered vertex}
Assume towards a contradiction that the $k \leq \tfrac{n^{0.52}}{10}$ hyperplanes defined by $v_1, \ldots, v_k \in \R^n$ and $\mu_1, \ldots, \mu_k \in \R$ form an essential cover of the hypercube for large $n$. 
Our goal is to locate a vertex $u$ that is not covered.
This is done in three phases. In each phase we ``discover'' some more coordinates in $u$. 

\subsection{Finding structure}
Let $V$ be the $k \times n$ matrix whose rows are $v_1, \ldots, v_k$. 
The analysis is based on a structural lemma for $V$ that is stated below.
Figure~\ref{fig:galaxy} may help in understanding the lemma.
The proof of the lemma is deferred to Section~\ref{decomposing_matrix_section}.
To state the lemma, we need the following definitions.
Denote by $\|v\|_0$ the sparsity of $v \in \R^n$;
i.e., $\|v\|_0$ is the number of non-zero entries in $v$. 

\begin{definition}[\cite{yehuda2021slicing}] \label{definition:many_scales}
Let $C_0 >1$ be the constant from Claim 10 in~\cite{yehuda2021slicing}.
A vector $v \in \R^n$ has many scales if $v$ can be partitioned into $S := \floor{n^{0.001}}$ vectors $v^{(1)}, v^{(2)}, \ldots, v^{(S)}$ such that
\begin{equation*}
    \|v^{(s)}\|_2 \geq C_0 \|v^{(s+1)}\|_2
\end{equation*}
for all $s < S$.
The smallest scale of $v$ is defined to be $\|v^{(S)}\|_2$. 
\end{definition}

\begin{lemma} \label{main_dcomposition}
Let $V \in \R^{k \times n}$ be the matrix defined above
(its rows are $v_1,\ldots,v_k$ and $k \leq \frac{n^{0.52}}{10}$). 
Then, there is a partition of the rows of $V$ to four parts $[k] = K_1 \cup K_2 \cup K_3 \cup K_4$ and a partition of the columns to three parts $[n] = N_1 \cup N_2 \cup N_3$ with $|N_1| \geq \tfrac{n}{2}$ such that the following hold: 
\begin{enumerate}
\item \label{itm:sparseCol}
Every column in $N_1 \cup N_2$ has sparsity at most $n^{0.04}$.
    \item For every row $i \in K_1$, the $N_1 \cup N_2$ coordinates of $v_i$ are zero.
    \item For every row $i \in K_2$, the $N_1$ coordinates of $v_i$ are zero, and there are at least $4|K_2|^2$ non-zero elements in the $N_2$ coordinates of $v_i$.
      \item For $i \in K_3$, let $v'_i$ be the projection of $v_i$
      to the $N_1$ coordinates. 
      The vector $v'_i$ is non-zero.
      Let $\phi_i>0$ be so that $\|\phi_i v'_i\|_2 = 1$.
      For every column $j \in N_1$,
    \begin{align} \label{small_l2}
        \sum_{i \in K_3} \phi_i^2 {v'}^{2}_{ij} < n^{-0.196} . 
    \end{align}
    In particular, item~\eqref{itm:sparseCol} implies
    \begin{align} \label{small_linf}
        \sum_{i \in K_3} |\phi_i v'_{ij}| < \sqrt{n^{0.04} n^{-0.196}} \leq n^{-0.078} .
    \end{align}
    \item For every row $i \in K_4$, the projection of $v_i$ to the $N_1 \cup N_2$ coordinates has many scales, its smallest scale is non-zero, and the position of the smallest scale contains $N_1$.
  \end{enumerate}
\end{lemma}

\begin{figure}[htp]
    \centering
    \includegraphics[width=11cm]{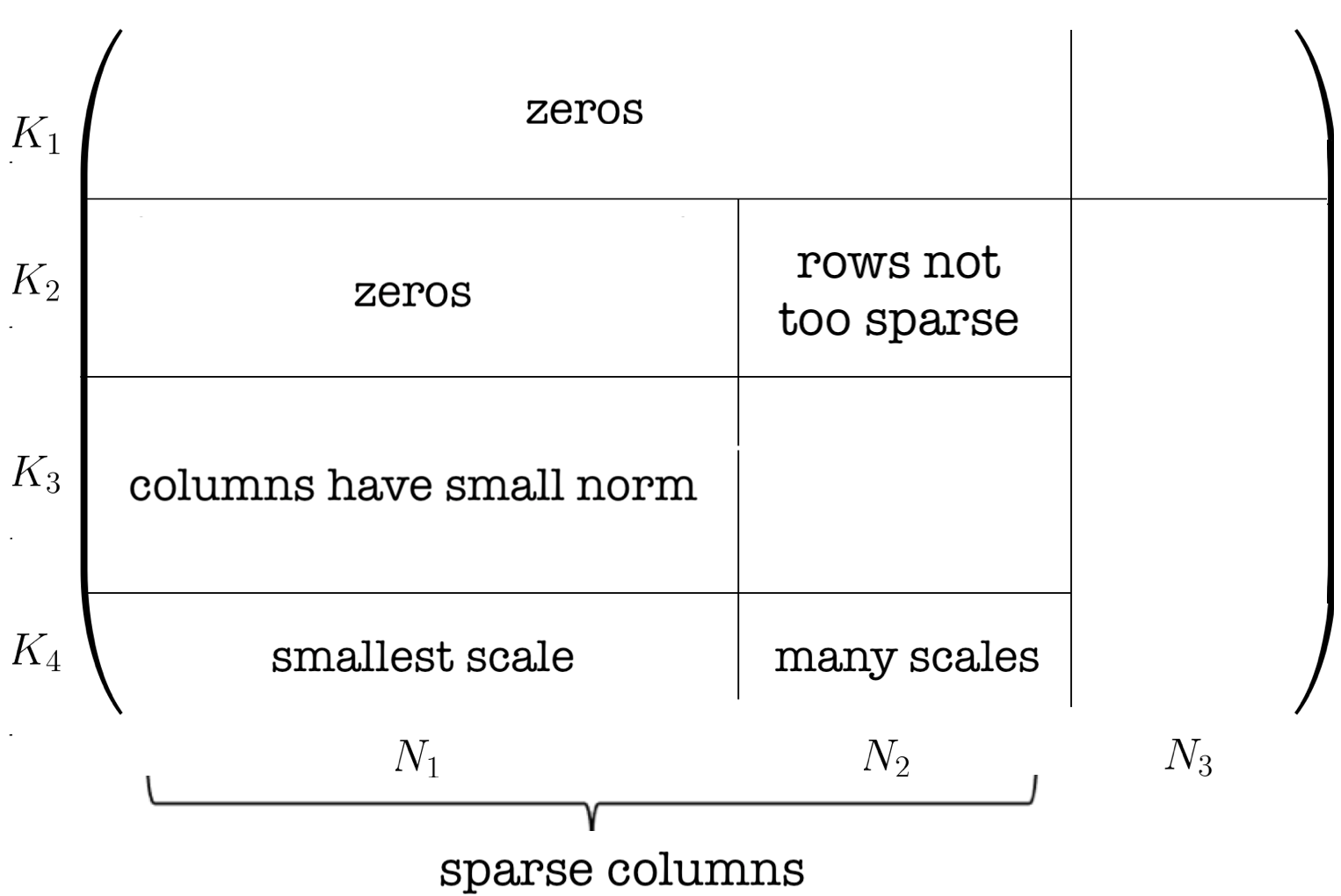}
    \caption{The decomposition of $V$.}
    \label{fig:galaxy}
\end{figure}

\subsection{Phase I}
In this phase, we fix the $N_3$ coordinates in $u$.
We know that $|N_3| < n$.
Condition (E\ref{essential-condition:all-vars}) implies that $K_1 \neq [k]$. 
Condition (E\ref{essential-condition:minimal}) implies
that there is a vertex $x$ that is not covered by the hyperplanes in $K_1$.
For every $j \in N_3$ set $u_j = x_j$.
The rest of the entries in $u$ shall be determined in later phases.
Denote the part of $u$ that was just fixed by $u^{(I)}$.
The $u^{(I)}$ part of $u$ deals with all hyperplanes in $K_1$.

\subsection{Phase II}
In this phase, we fix the $N_2$ coordinates in $u$.
This is done with the aid of randomness. 
The randomness simultaneously deals with rows in $K_2$ and in $K_4$.
We denote by $v'_i$ the projection of $v_i$ to the $N_1$ coordinates,
by $v''_i$ the projection to the $N_2$ coordinates,
and by $v^{(I)}_i$ the projection to coordinates in $N_3$.

\begin{claim} \label{claim:n2}
There is $x'' \in \{\pm 1\}^{N_2}$ so that
\begin{enumerate}
    \item For every $i \in K_2$, 
    \begin{equation*}
        \langle x'', v_i'' \rangle + \langle u^{(I)}, v_i^{(I)} \rangle \neq \mu_i .
    \end{equation*}
    \item For every $i \in K_4$ and $x' \in \{\pm 1\}^{N_1}$, 
    \begin{equation*}
        |\langle x'', v_i'' \rangle + \langle u^{(I)}, v_i^{(I)} \rangle - \mu_i| > |\langle x', v_i' \rangle|  .
    \end{equation*}
\end{enumerate}
\end{claim}

Choose the $N_2$ coordinates of $u$ to be $x''$.
Denote this part of $u$ by $u^{(II)}$.
The first condition implies that $u^{(II)}$ deals with all rows in $K_2$.
The second condition implies that it deals with all rows in $K_4$,
regardless of the future choice of the $N_1$ coordinates in $u$.

In order to prove the claim, we need the following two anti-concentration results.
The first is known as the Littlewood-Offord lemma~\cite{littlewood1939number}
and was proved by Erd\H os~\cite{erdos1945lemma}.
The second appears in~\cite{yehuda2021slicing}.
We denote by $x \sim \{\pm 1\}^n$ a uniformly random vertex
of the cube.

\begin{claim}[\cite{littlewood1939number,erdos1945lemma}] \label{anti-concentration:sparse-vectors}
For all $v \in \R^n$ so that $\|v\|_0>0$ and $a \in \R$, 
$$\Pr_{x \sim \{\pm 1\}^{n}} [\langle x, v \rangle = a] \leq \frac{1}{\sqrt{\|v\|_0}}.$$
\end{claim}


\begin{lemma}[\cite{yehuda2021slicing}]
\label{lemma12_previous}
There is a constant $C_1 > 1$ so that the following holds. 
If $v \in \R^n$ has many scales and its smallest scale is $\delta > 0$, then for every $a \in \R$ and $b \geq 2$ we have
\begin{equation*}
    \Pr_{x \sim \{\pm 1\}^n}[|\langle x,v \rangle - a| \leq b \delta] < C_1 \exp\Big(-\frac{S}{C_1} + C_1\log(b)\Big) .
\end{equation*}
\end{lemma}

\begin{proof}[Proof of Claim \ref{claim:n2}]
Pick $x''$ uniformly at random. 
First, consider $i \in K_2$. 
By Claim~\ref{anti-concentration:sparse-vectors},
because $\|v''_i\|_0 \geq 4|K_2|^2$,
\begin{equation*}
    \Pr_{x''}[ \langle x'', v_i'' \rangle + \langle u^{(I)}, v_i^{(I)} \rangle = \mu_i]  \leq \frac{1}{2 |K_2|} .
\end{equation*}

Second, consider $i \in K_4$.
Row $v_i$ has many scales and $v'_i$ is part of its smallest scale. 
Let $\alpha$ be the vector consisting of the largest $S-1$ scales of $v_i$. 
The vector of the smallest scale in $v_i$ has two parts: $v'_i$ and the part outside the $N_1$ columns which we denote by $\beta$. 
Partition $x''$ to $x_\alpha$, $x_\beta$ accordingly. 
Let $\delta>0$ be the smallest scale of $\alpha$. 
By Cauchy-Schwartz, for all $x' \in \{\pm 1\}^{N_1}$,
\begin{equation*}
    |\langle x', v'_i \rangle| \leq \sqrt{n} \|v'_i\|_2 < \sqrt{n} \delta .
\end{equation*}
By Lemma \ref{lemma12_previous}, conditioned on the value of $x_\beta$, 
\begin{align*}
     \Pr_{x_\alpha}[|\langle x'', v''_i \rangle + \langle u^{(I)}, v_i^{(I)} & \rangle - \mu_i| \leq
    \sqrt{n} \delta ] \leq C_1 \exp \Big(-\frac{S-1}{C_1} + C_1 \log n\Big) .
\end{align*}
The same bound holds when we average over $x_\beta$ as well.

Finally, the union bound over $K_2 \cup K_4$ completes the proof.
\end{proof}

\subsection{Phase III}
It remains to choose the $N_1$ coordinates in $u$.
Let $\tilde u \in \{\pm 1\}^{N_2 \cup N_3}$ denote the part of $u$ that was already fixed
(it is comprised of $u^{(I)}$ and $u^{(II)}$).
Let $\tilde v_i$ be projection of $v_i$
to the $N_2 \cup N_3$ coordinates.

\begin{claim}
\label{clm:finalone}
There is $x' \in \{\pm 1\}^{N_1}$ so that for all $i \in K_3$,
$$\langle x' , v'_i \rangle + \langle \tilde u, \tilde v_i \rangle \neq \mu_i .$$
\end{claim}

Before proving the final claim, let us see how it completes the proof
of the theorem.
The final part $u^{(III)}$ of $u$ is $x'$.
The vertex $u$ is indeed not covered by any of the hyperplanes:
All hyperplanes in $K_1$ are avoided due to~$u^{(I)}$,
because the rest of their entries are zero.
All hyperplanes in $K_2$ are avoided due to~$u^{(II)}$,
because the rest of their entries are zero.
All hyperplanes in $K_3$ are avoided due to~$u^{(III)}$.
All hyperplanes in $K_4$ are avoided due to~$u^{(II)}$ as well,
because $u^{(III)}$ has no real affect.

A key piece of the argument is a lemma that Ball isolated
from Bang's solution
of Tarski's plank problem~\cite{bang1951solution,ball1991plank}.
\begin{lemma}[\cite{bang1951solution,ball1991plank}] \label{lemma:bang}
Let $M$ be a $k \times k$ real symmetric matrix with ones on the diagonal.
Let $\gamma_1, \ldots, \gamma_k \in \R$ and $\theta \geq 0$. 
Then, there is $\epsilon \in \{\pm 1\}^k$ so that for every $i \in [k]$,
\[|\theta (M \epsilon)_i - \gamma_i| \geq \theta . \]
\end{lemma}

\begin{proof}[Proof of Claim~\ref{clm:finalone}]
For $i \in K_3$, 
let $\phi_i>0$ be so that $\|\phi_i v'_i\|_2 = 1$
and let
\[ \gamma_i = \phi_i (\mu_i - \langle \tilde u, \tilde v_i \rangle).\]
Let $V'$ be the matrix with rows $\phi_i v'_i \in \R^{N_1}$
for $i \in K_3$.
Consider the $K_3 \times K_3$ matrix $M = V'V'^{T}$.
Let
\[\theta = n^{0.078} .\]
By Bang's lemma, there is $\epsilon \in \{\pm 1\}^{K_3}$ so that for each $i \in K_3$, 
\begin{align}
\label{eqn:bang}
    |\langle \phi_i v'_i , z \rangle - \gamma_i| \geq \theta ,
\end{align}
where 
\[z = \theta V'^T \epsilon . \]
By (\ref{small_linf}), 
\[\|z\|_{\infty} \leq 1 . \]
We round $z$ to a vertex of the cube $\{\pm 1\}^{N_1}$ in two steps. 
The first step uses linear algebra as in~\cite{yehuda2021slicing}.
\begin{claim}[\cite{yehuda2021slicing}]
\label{clm:linalg}
There is $w \in \R^{N_1}$ so that the following hold: 
\begin{enumerate}
    \item For each $ i \in K_3$, we have $\langle w,v_i' \rangle = \langle z,v_i' \rangle$.
    \item $\|w\|_{\infty} \leq 1$. 
    \item Let $N_0$ be the set of coordinates in $w$ that are not $\pm 1$.
    Then $|N_0| \leq k$.
\end{enumerate}
\end{claim}

If $N_0$ is empty then we are done, in light of~\eqref{eqn:bang}
and because $w \in \{\pm 1\}^{N_1}$
is so that
$$\langle w , v'_i \rangle + \langle \tilde u, \tilde v_i \rangle = \mu_i 
\quad \Leftrightarrow \quad 
\langle w , \phi_i v'_i \rangle = \gamma_i .$$
Otherwise, we need to round the coordinates in $N_0$ to $\pm 1$.
This is done using extra randomness.
Let $\delta \in \R^{N_1}$ be a random vector distributed as follows.
Its coordinates are independent so that $\delta_j + w_j$ takes values in $\{\pm 1\}$ and  $\E \delta_j = 0$.
Its marginals can be computed by
\begin{equation*}
    w_j = \E[\delta_j + w_j] = 2\Pr[\delta_j = 1-w_j] - 1
\end{equation*}
so that 
\begin{equation*}
    \Pr[\delta_j = 1-w_j] = \frac{1 + w_j}{2} 
\end{equation*}
and
\begin{align*}
    \E[\delta_j^2] &= \frac{1 + w_j}{2} (1-w_j)^2 + \frac{1-w_j}{2} (-1-w_j)^2 
    = 1 - w_j^2 .
\end{align*}
Consider the random vertex 
$$x' = \delta + w \in \{\pm 1\}^{N_1}.$$
The rest of the analysis is split between two cases. 
For each $i \in K_3$, let 
\[\sigma_i^2 = \sum_{j \in N_1} (1-w_j^2) \phi_i^2 {v'_{ij}}^2 . \]
First, consider some $i$ so that $\sigma_{i}^2 \leq n^{0.151}$; if there are no such $i$'s then continue to case two. 
We use the following concentration of measure.
\begin{theorem}[Bernstein] Let $z_1, \ldots, z_{\ell}$ be independent random variables with mean zero that are almost surely at most two in absolute value. Let $\sigma^2 = \sum_{j} \E[z_j^2]$. Then, for all $t > 0$, we have 
$\Pr\big[\sum_{j} z_j \geq t\big] \leq \exp\big(-\frac{t^2}{2\sigma^2 + 2t}\big) $.
\end{theorem}
By~\eqref{eqn:bang}, by the choice of $\theta$
and by Claim~\ref{clm:linalg},
\begin{align*}
    \Pr[\langle x', \phi_i v_i' \rangle = \gamma_i ]
    =  \Pr[\langle \delta , \phi_i v_i' \rangle = \gamma_i - \langle z , \phi_i v_i' \rangle] \leq \Pr[|\langle  \delta, \phi_i v_i' \rangle| \geq n^{0.078}] .
\end{align*}
By Bernstein's inequality, 
\begin{align*}
    \Pr[|\langle  \delta, \phi_i v_i' \rangle| \geq n^{0.078}]
    &\leq 2 \exp\Big(-\frac{n^{0.154}}{2n^{0.151} + 2n^{0.078}}\Big) .
\end{align*}
This probability is small enough so that we can safely apply the union bound
over all such $i$'s.

We now move to the second and final case. 
We harness anti-concentration one more time
(for the proof see Section~\ref{section:anti-concentration}).

\begin{claim}  \label{anti_chain_vertex}
There is a constant $C>0$ so that the following holds.
Let $P$ be a product distribution on $\{\pm 1\}^n$.
Let $\sigma^2$ be the variance of $\sum_j z_j$ for $z \sim P$.
Assume that $\sigma^2 > 0$.
For every vector $v \in \R^n$ so that $\|v\|_0 = n$ and $\gamma \in \R$,
we have $\Pr_{z \sim P}[\langle v,z \rangle = \gamma] \leq \frac{C}{\sigma}$.
\end{claim}

Fix some $i$ so that $\sigma_i^2 > n^{0.151}$. 
Let $J$ be the set of non-zero coordinates in $v'_i$.
The variance of $\sum_{j \in J} x'_j$ is
$\sigma_J^2 = \sum_{j \in J} \tfrac{1-w_j^2}{4}$.
Using~\eqref{small_l2},
\begin{align}
     4 n^{-0.196} \sigma_J^2 \geq \sum_{j \in J} (1-w_j^2) \phi^2_i {v'_{ij}}^2 =
     \sigma_i^2 > n^{0.151} .
\end{align}
By Claim~\ref{anti_chain_vertex}, the probability that 
$\langle x', \phi_i v'_i \rangle = \gamma_i$ is at most $\frac{2 C}{n^{0.1735}}$.

We need to justify the final union bound.
By \eqref{small_l2} again, 
\begin{align*}
    \sum_{i \in K_3} \sigma_i^2 = \sum_{j \in N_1} (1 - w_j^2) \sum_{i\in K_3} \phi_i^2 {v'_{ij}}^2 < n^{-0.196} |N_0| \leq n^{0.324} .
\end{align*}
It follows that the number of $i$'s with $\sigma_i^2 > n^{0.151}$ is at most
$n^{0.324-0.151} \leq n^{0.173}$.
The probability that 
there exists $i$ with $\sigma_i^2 > n^{0.151}$ so that
$\langle x', \phi_i v'_i \rangle = \gamma_i$ is at most
\begin{equation*}
    n^{0.173} \cdot \tfrac{2C}{ n^{0.1735}} \leq n^{-0.0004}  . \qedhere
\end{equation*}
\end{proof}

\section{Decomposing the matrix} \label{decomposing_matrix_section}

For the proof of Lemma~\ref{main_dcomposition},
we need a version of Lemma 13 in~\cite{yehuda2021slicing}.
For completeness, we include its proof below.

\begin{lemma}[\cite{yehuda2021slicing}]
\label{lemma:decomposing}
Let $V \in \R^{k \times n}$ be a matrix so that $k \leq n^{0.52}$.
Then, there is a partition of the rows of $V$ to two parts $[k] = L_1 \cup L_2$ and a partition of the columns to two parts $[n] = M_1 \cup M_2$ with $|M_2| \leq n^{0.7171}$ such that the following hold: 
\begin{enumerate}
    \item \label{itm:oneLem}
    Let $V' = (v'_{ij})$ be the submatrix of $V$ defined by rows in $L_1$ and columns in $M_1$; normalize the non-zero rows in $V'$ to have $\ell_2$-norm one.
    For every column $j \in M_1$, we have $\sum_{i \in L_1} {v'}^{2}_{ij} < n^{-0.196}$. 
    \item Every row $i \in L_2$ has many scales, and the position of its smallest scale contains the $M_1$ columns. 
\end{enumerate}
\end{lemma}

\begin{proof}
We use the following terminology.
The mass of a vector $w$ is $\|w\|_2^2$.
Initialize $L_1 = [k], L_2 = \emptyset, M_1 = [n]$ and $M_2 = \emptyset$.
If there is no column $j \in M_1$ with mass
\begin{equation*}
    \sum_i {v_{ij}}^2 \geq n^{-0.1961} ,
\end{equation*}
then we are done.
Otherwise, there are such columns. 
Start moving them one-by-one from $M_1$ to $M_2$.
When we move a column, the norm inside $M_1$ of each row changes, but we do not immediately renormalize it.
 
Let $\tau > 0$ be so that 
$\tfrac{1 - \tau}{\tau} = C^2_0$, where $C_0$ is from Definition \ref{definition:many_scales}. 
If for a given row $i \in L_1$, after the removal of a column, the total current mass inside $M_1$ becomes less than $\tau$, then mark a ``drop'' for row $i$.
At the same time, 
if row $i$ is non-zero inside $M_1$,
renormalize it so that its norm inside $M_1$ is one.

When a drop occurs, we get one more scale for $v_i$, by the choice of $\tau$. 
If a row is dropped more than $S = \floor{n^{0.001}}$ times, then move it from $L_1$ to $L_2$.
Each row in $L_2$ has many scales and the position of its minimum scale contains $M_1$.

Let $\sigma_t$ be the mass of the column that was added to $M_2$ at time $t$, with the normalization at time $t$. 
So, for all $t$, 
we have $\sigma_t \geq n^{-0.1961}$.
If there are $T$ iterations, then 
\begin{equation*}
    \sum_t \sigma_t \geq n^{-0.1961} T .
\end{equation*}
Let $\sigma_{t,i}$ be the contribution of row $i$ to $\sigma_t$; it is zero if the row is in $L_2$ at time $t$. 
Because every row is dropped at most $S$ times, for each row $i$, we have $\sum_t \sigma_{t,i} < S$. 
So, 
\begin{equation*}
    \sum_t \sigma_t  < Sk \leq n^{0.521} .
\end{equation*}
After $T \leq n^{0.7171}$ steps, we get a matrix $V'$ so that the mass of every column $j \in M_1$ in it is
\begin{equation*}
    \sum_{i \in L_1} {v'_{ij}}^2 < n^{-0.1961} .
\end{equation*}
The rows of this matrix are not yet normalized. 
The norm of each row is either zero or at least~$\sqrt{\tau}$.
The renormalization can not increase the norm of the columns by more than $\frac{1}{\sqrt{\tau}}$.
After the final renormalization we get
\begin{equation*}
    \sum_{i \in L_1} {v'_{ij}}^2 < \frac{n^{-0.1961}}{\sqrt{\tau}} < n^{-0.196} .
   \qedhere
\end{equation*}
\end{proof}

\begin{proof}[Proof of Lemma \ref{main_dcomposition}]
We start by using the main lemma from~\cite{linial2005essential} which is proved using polynomial algebra.
The lemma relies on the fact that $v_1,\mu_1,\ldots,v_k,\mu_k$
define an essential cover.

\begin{lemma}[\cite{linial2005essential}] \label{lemma:nati}
For all $i \in [k]$, we have
$\|v_i\|_0 < 2 k$. 
\end{lemma}

Lemma~\ref{lemma:nati} implies
that the total number of non-zero entries in $V$ is at most
$2k^2 \leq \frac{n^{1.04}}{50}$.
Let $N_4$ be the set of columns with more than $n^{0.04}$ non-zero entries.
Thus, $|N_4| \leq \frac{n}{50}$.
%

Construct a nested sequence of matrices $V^{(0)},V^{(1)},\ldots,V^{(T)}$
by repeatedly applying Lemma \ref{lemma:decomposing} as follows. 
The matrix $V^{(0)}$ is obtained from $V$ by removing
the $N_4$ columns. All columns in $V^{(0)}$ are therefore sparse.
Apply Lemma \ref{lemma:decomposing} to $V^{(0)}$
to get $L_1,L_2$ and $M_1,M_2$.
Let $Z_1$ be the set of rows in $V^{(0)}$
so that their $M_1$ part is zero.
That is, for $i \in Z_1$ the projection $v_i|_{M_1}$ of $v_i$
to the $M_1$ coordinates is zero.
Rows in $Z_1$ may come from $L_1$ as well as $L_2$.
Let $k_1 = |Z_1|$ and $n_1 = |M_2|$.
Check if the following two conditions hold:
\begin{align}
    k_1 > n_1^{0.332}  \label{condition1:large_k}
\end{align}
and there is $i^* \in Z_1$ so that
the projection of $v_{i^*}$ to coordinates in $M_2$
has few non-zero entries:
\begin{align}
    \|v_{i^*}|_{M_2}\|_0 \leq 4 k_1^2 . \label{condition2: sparse_V}
\end{align}

\begin{description}
\item[If~\eqref{condition1:large_k} holds]
Let $L'_1 = [k] \setminus Z_1$. 
Let $V^{(1)}$ be the $L'_1 \times M_1$ submatrix of $V^{(0)}$,
and apply the same procedure on $V^{(1)}$ to construct the matrices $V^{(2)},\ldots,V^{(T)}$.

\item[If~\eqref{condition1:large_k} does not hold
but \eqref{condition2: sparse_V} holds] 
Let $L'_1 = [k] \setminus \{i^*\}$
and let $M'_1$ be the zero coordinates of $v_{i^*}$.
Let $V^{(1)}$ be the $L'_1 \times M'_1$ submatrix of $V^{(0)}$
and apply the procedure on $V^{(1)}$ to get $V^{(2)},\ldots,V^{(T)}$.

\item[If both conditions do not hold] Output the $(L_1 \setminus Z_1) \times M_1$ submatrix of the input matrix $V^{(0)}$.

\end{description}

We got a sequence of nested matrices $V^{(0)},\ldots,V^{(T)}$
and a corresponding sequence $(n_1,k_1),\ldots,(n_T,k_T)$.
The matrix $V^{(T)}$ satisfies both~\eqref{condition1:large_k}
and~\eqref{condition2: sparse_V} but
a priori it may be empty.

\subsubsection*{The partitions}
It can be helpful to recall Figure~\ref{fig:galaxy}.

Set $N_3$ to be the set of columns that are not in $V^{(T-1)}$;
this includes all the dense columns $N_4$,
so that all columns not in $N_3$ are sparse.
Set $N_2$ to be the set of columns in $V^{(T-1)}$ but not in $V^{(T)}$.
And set $N_1$ to be the set of columns in $V^{(T)}$.

Set $K_1$ to be the set of rows outside $V^{(T-1)}$;
the $K_1 \times (N_1 \cup N_2)$ submatrix of $V$ is zero.
Set $K_2$ to be $Z_T$;
the $K_2 \times N_1$ submatrix of $V$ is zero,
and no row in the $K_2 \times N_2$ submatrix of $V$ is sparse
(condition~\eqref{condition2: sparse_V}).
Set $K_3$ to be the rows inside $V^{(T)}$; 
the norm of the columns in the $K_3 \times N_1$
submatrix of $V$ can be bounded via item~\ref{itm:oneLem}
in Lemma~\ref{lemma:decomposing}.
Set $K_4$ to be the non-zero rows with many scales 
in $V^{(T-1)}$ and not in $V^{(T)}$.
No row in the $(K_3 \cup K_4) \times N_1$
submatrix of $V$ is zero, by construction of $Z_T$.

\subsubsection*{The size of $N_1$}
We already know that $N_4$ is small.
It remains to show that $N_2 \cup N_3$ is small as well.
The total number of columns removed due to condition \eqref{condition1:large_k} is smaller than $\frac{n}{8}$, because 
otherwise we get a contradiction via the following claim.

\begin{claim}
Let $\lambda \in (0,1]$ and $\xi_1,\ldots,\xi_m \in [0,A]$ be so that
$\sum_{t \in [m]} \xi_t \geq B$ then
$\sum_{t \in [m]} \xi_t^\lambda \geq \tfrac{B}{2A} A^\lambda$.
\end{claim}

\begin{proof}
Partition $[m]$ to sets $R_1,R_2,\ldots,R_Q$ so that for 
all $q \in [Q]$ we have $A \leq \sum_{t \in R_q} \xi_t \leq 2A$
and $Q \geq \tfrac{B}{2A}$.
Use the inequality
$\alpha^\lambda + \beta^\lambda \geq (\alpha+\beta)^\lambda$
that is valid for all $\alpha,\beta \geq 0$.
For each $q$, we have $\sum_{t \in R_q} \xi_t^\lambda
\geq A^\lambda$. Summing over all $q$ completes the proof.
\end{proof}

Denote by $\tau \subseteq [T]$ the set of times at which condition
\eqref{condition1:large_k} is applied.
Summing over these times we get
\begin{align*}
    k \geq  \sum_{t \in \tau} k_t \geq  \sum_{t \in \tau} n_t^{0.332} .
\end{align*}
By Lemma~\ref{lemma:decomposing}, we know that $n_t \leq n^{0.7171}$ for all $t$.
The condition $\sum_{t \in \tau} n_t \geq \tfrac{n}{8}$ hence implies
\begin{align*}
 \sum_{t \in \tau} n_t^{0.332} \geq \frac{n}{16 n^{0.7171}} \cdot (n^{0.7171})^{0.332} \geq n^{0.52} ,
\end{align*}
which is a contradiction.

The total number of columns removed when \eqref{condition1:large_k} does not hold and \eqref{condition2: sparse_V} holds can be bounded as follows.
Because the number of $t \not \in \tau$ is at most $k$,
and because $n_t \leq n^{0.7171}$ for all $t$,
the number of columns is at most
\[\sum_{t \not \in \tau} 4 k^2_t \leq \sum_{t \not \in \tau} 4 n_t^{0.664}
\leq  4 n^{0.477} k
< \frac{n}{8} . \qedhere \]
\end{proof}

\section{Antichains} \label{section:anti-concentration}

Here we prove the last anti-concentration bound we need
(Claim~\ref{anti_chain_vertex}).
The set~$\{\pm 1\}^n$ has a standard partial order:
$x \leq y$ iff $x_j \leq y_j$ for all $j \in [n]$.
An antichain in $\{\pm 1\}^n$ is a 
set of incomparable vertices.
The following theorem 
states that antichains do not have large mass in general product measures.\footnote{The assumption that $P$ is non-trivial
in the theorem from~\cite{yehuda2021slicing}
can be removed 
by restricting to the non-trivial part of $P$.}

\begin{theorem}[\cite{yehuda2021slicing}] \label{theorem:product_dist}
There is a constant $C>0$ so that the following holds. 
Let $P$ be a product distribution on $\{\pm 1\}^n$.
Assume that the variance $\sigma_P^2$ of $\sum_j z_j$ for $z \sim P$ is positive.
Then, for every antichain $A \subset \{\pm 1\}^n$, we have
$\Pr_{z \sim P}[z \in A] \leq \frac{C}{\sigma_P}$.
\end{theorem}

\begin{proof}[Proof of Claim~\ref{anti_chain_vertex}]
Let $s \in \{\pm 1\}^n$ be so that $s_j v_j>0$ for all $j$.
Let $P'$ be the distribution of $z' = z \odot s$ where $z \sim P$
and $\odot$ is coordinate-wise multiplication.
The distribution $P'$ is a product distribution and
$\sigma_P = \sigma_{P'}$.
Because $\langle z' , v \odot s\rangle = \langle z, v \rangle$,
we need to bound $\Pr[\langle z' , v \odot s\rangle = \mu]$.
The set $A$ of vertices $x$ so that $\langle x, v \odot s\rangle=\mu$
is an antichain.
The reason is that if $x  < y$
then $\langle x, v \odot s\rangle < \langle y, v \odot s\rangle$.
Theorem~\ref{theorem:product_dist} completes the proof.
\end{proof}

\subsection*{Acknowledgement}

We thank Nati Linial for helpful conversations.

\bibliographystyle{amsplain}
\bibliography{bib.bib}

\end{document}